\documentclass[11pt]{amsart}
\usepackage{amssymb,amsmath,amsfonts,amscd,euscript}

\newcommand{\nc}{\newcommand}

\numberwithin{equation}{section}
\newtheorem{thm}{Theorem}[section]
\newtheorem*{thm*}{Theorem}
\newtheorem{prop}[thm]{Proposition}
\newtheorem{lem}[thm]{Lemma}

\newtheorem{cor}[thm]{Corollary}
\newtheorem*{cor*}{Corollary}
\theoremstyle{remark}
\newtheorem{rem}[thm]{Remark}

\newtheorem{dfn}[thm]{Definition}

\newcount\cols {\catcode`,=\active\catcode`|=\active
 \gdef\Young(#1){\hbox{$\vcenter
 {\mathcode`,="8000\mathcode`|="8000
  \def,{\global\advance\cols by 1 &}%
  \def|{\cr
        \multispan{\the\cols}\hrulefill\cr
        &\global\cols=2 }%
  \offinterlineskip\everycr{}\tabskip=0pt
  \dimen0=\ht\strutbox \advance\dimen0 by \dp\strutbox
  \halign
   {\vrule height \ht\strutbox depth \dp\strutbox##
    &&\hbox to \dimen0{\hss$##$\hss}\vrule\cr
    \noalign{\hrule}&\global\cols=2 #1\crcr
    \multispan{\the\cols}\hrulefill\cr%
   }
 }$}}
 \gdef\Skew(#1:#2){\hbox{$\vcenter
 {\mathcode`,="8000\mathcode`|="8000
  \dimen0=\ht\strutbox \advance\dimen0 by \dp\strutbox
  \def\boxbeg{\vbox
    \bgroup\hrule\kern-0.4pt\hbox to\dimen0\bgroup\strut\vrule\hss$}%
  \def\boxend{$\hss\egroup\hrule\egroup}%
  \def,{\boxend\boxbeg}%
  \def|##1:{\boxend\vrule\egroup\nointerlineskip\kern-0.4pt
    \moveright##1\dimen0\hbox\bgroup\boxbeg}%
  \def\\##1\\##2:{\boxend\vrule\egroup\nointerlineskip\kern-0.4pt
    \kern ##1\dimen0\moveright##2\dimen0\hbox\bgroup\boxbeg}%
  \moveright#1\dimen0\hbox\bgroup\boxbeg#2\boxend\vrule\egroup
 }$}} }

\def\smallsquares
 {\textfont0=\scriptfont0 \scriptfont0=\scriptscriptfont0
  \textfont1=\scriptfont1 \scriptfont1=\scriptscriptfont1
  \setbox0=\hbox{$($}
  \setbox\strutbox=\hbox{\vrule width 0pt height\ht0 depth\dp0 }
 }

\nc{\gl}{\mathfrak{gl}}
\nc{\GL}{\mathfrak{GL}}
\nc{\g}{\mathfrak{g}}
\nc{\gh}{\widehat\g}
\nc{\h}{\mathfrak{h}}
\nc{\n}{\mathfrak{n}}
\nc{\la}{\lambda}
\nc{\C}{\mathbb C }
\nc{\D}{\mathbb D }
\nc{\Z}{\mathbb Z }
\nc{\N}{\mathbb N }
\nc{\R}{\mathbb R }
\nc{\Q}{\mathbb Q }
\nc{\al}{\alpha }
\nc{\bs}{{\bf s}}
\nc{\bbs}{\ol{\bf s}}
\nc{\bt}{{\bf t}}
\nc{\br}{{\bf r}}
\nc{\bp}{{\bf p}}
\nc{\bP}{{\bf P}}
\nc{\bm}{{\bf m}}
\nc{\ba}{\ol{\alpha} }
\nc{\bj}{{\bf j}}
\nc{\bk}{{\bf k}}
\nc{\bfa}{{\bf a}}
\nc{\bfb}{{\bf b}}

\nc{\om}{\omega}
\nc{\ord}{\succ}
\nc{\ta}{\theta}
\nc{\ve}{\varepsilon}
\nc{\ch}{{\mathop {\rm ch}}}
\nc{\Tr}{{\mathop {\rm Tr}\,}}
\nc{\Id}{{\mathop {\rm Id}}}
\nc{\ad}{{\mathop {\rm ad}}}
\nc{\bra}{\langle}
\nc{\ket}{\rangle}
\nc{\bx}{{\bf x}}
\nc{\pa}{\partial}
\nc{\ld}{\ldots}
\nc{\cd}{\cdots}
\nc{\hk}{\hookrightarrow}
\nc{\T}{\otimes}
\newcommand{\bea}{\begin{equation}}
\newcommand{\ena}{\end{equation}}
\newcommand{\be}{\begin{equation*}}
\newcommand{\en}{\end{equation*}}
\nc{\gr}{\mathrm{gr}}
\nc{\ov}{\overline}

\nc{\msl}{\mathfrak{sl}}
\nc{\msp}{\mathfrak{sp}}
\nc{\mgl}{\mathfrak{gl}}
\nc{\U}{\hbox{\rm U}}
\nc{\V}{\EuScript V}

\newcommand{\ol}{\overline}
\newcommand{\bc}{{\mathbb C}}

\newcommand{\fn}{{\mathfrak n}}

\newcommand{\lam}{\lambda}

\def\gr{\operatorname{gr}}

\def\wt{\operatorname{wt}}

\begin{document}

\title[PBW filtration and bases]
{PBW filtration  and bases for  symplectic Lie algebras}

\author{Evgeny Feigin, Ghislain Fourier and Peter Littelmann}
\address{Evgeny Feigin:\newline
Department of Mathematics, University Higher School of Economics,
20 Myasnitskaya st, 101000, Moscow, Russia
{\it and }\newline
Tamm Department of Theoretical Physics,
Lebedev Physics Institute,\newline
Leninsky prospect, 53,
119991, Moscow, Russia}
\email{evgfeig@gmail.com}
\address{Ghislain Fourier:\newline
Mathematisches Institut, Universit\"at zu K\"oln,\newline
Weyertal 86-90, D-50931 K\"oln,Germany
}
\email{gfourier@math.uni-koeln.de}
\address{Peter Littelmann:\newline
Mathematisches Institut, Universit\"at zu K\"oln,\newline
Weyertal 86-90, D-50931 K\"oln,Germany
}
\email{littelma@math.uni-koeln.de}

\begin{abstract}
We study the PBW filtration on the highest weight representations $V(\la)$ of
$\msp_{2n}$. This filtration is induced by the standard degree filtration on
$U(\n^-)$. We give a description
of the associated graded $S(\n^-)$-module $gr V(\la)$ in terms of generators and
relations. We also construct a basis of $gr V(\la)$. As an application we
derive a graded combinatorial formula for the character of $V(\la)$ and obtain
a new class of bases of the modules $V(\la)$.
\end{abstract}
\maketitle

\section*{Introduction}
In this paper we continue the study of the PBW filtration on irreducible representations
of simple Lie algebras initiated in \cite{FFoL}. The goal of this paper is to develop
the theory of PBW-graded modules for symplectic Lie algebras $\msp_{2n}$. We start with recalling
the definition of the PBW filtration.

Let $\g$ be a simple Lie algebra and let $\g=\n^+\oplus\h\oplus \n^-$ be a Cartan decomposition.
For a dominant integral $\la$ we denote by $V(\la)$ the irreducible $\g$-module with
highest weight $\la$. Fix a highest weight vector $v_\la\in V(\la)$.
Then $V(\la)=\U(\n^-)v_\la$, where $\U(\n^-)$ denotes the universal
enveloping algebra of $\n^-$. The degree filtration $\U(\n^-)_s$ on $\U(\n^-)$ is defined by:
$$
\U(\n^-)_s=\mathrm{span}\{x_1\dots x_l:\ x_i\in\n^-, l\le s\}.
$$
In particular, $\U(\n^-)_0=\C$ and $gr \U(\n^-)\simeq S(\n^-)$, where $S(\n^-)$
denotes the symmetric algebra over $\n^-$. The filtration of $\U(\n^-)$
by the subspaces $\U(\n^-)_s$ induces a filtration of $V(\la)$ by the subspaces
$V(\la)_s$:
$$
V(\la)_s=\U(\n^-)_sv_\la.$$
We call this filtration the
PBW filtration.
The central objects of our paper are the associated graded spaces
$gr V(\la)$ as $S(\n^-)$-modules for $\g$ of type ${\tt C}_n$.

We note that
$gr V(\la)=S(\n^-)v_\la$ is a cyclic $S(\n^-)$-module. So one has
$$
gr V(\la)\simeq S(\n^-)/I(\la),
$$
for some ideal $I(\la)\subset S(\n^-)$. For example, for any simple root $\al_i$
the power $f_{\al_i}^{(\la,\al_i)+1}$ of a root vector $f_{\al_i}\in\n^-_{-\alpha_i}$
belongs to $I(\la)$
since $f_{\al_i}^{(\la,\al_i)+1}v_\la=0$ in $V(\la)$. To
describe $I(\la)$ explicitly,
we prepare some notations.
All positive roots of $\msp_{2n}$
can be divided into two groups:
\begin{gather*}
\al_{i,j}=\al_i+\al_{i+1}+\dots +\al_j,\ 1\le i\le j\le n,\\
\alpha_{i, \ol{j}} = \alpha_i + \alpha_{i+1} + \ldots +
\alpha_n + \alpha_{n-1} + \ldots + \alpha_j, \ 1\le i\le j\le n.
\end{gather*}
In particular, $\al_{1,\ol{1}}$ is the highest root.
Consider the action of the opposite subalgebra
$\n^+$ on $V(\la)$. It is easy to see that $\n^+ V(\la)_s\hk V(\la)_s$, so we
obtain the structure of an $U(\n^+)$-module on $gr V(\la)$ as well.
We show:
\vskip 5pt\noindent 
{\bf Theorem~A.}\ {\it The ideal $I(\la)$ is generated as $S(\n^-)$-module by the subspace }
$$\U(\n^+)\circ \mathrm{span}
\{f_{\al_{i,j}}^{(\la,\al_{i,j})+1}, 1\le i\le j\le n-1,\
f_{\al_{i,\ol{i}}}^{(\la,\al_{i,n})+1}, 1\le i\le n\}.
$$
\par\vskip 5pt
\noindent
Theorem $A$ should be understood as a commutative analogue
of the well-known description of $V(\lam)$ as the quotient
$$
V(\la)\simeq \U(\n^-)/\langle  f_{\al_i}^{(\la,\al_i)+1}, 1\le i\le n \rangle
$$
(see for example \cite{H}).

Our second problem (closely related to the first one)
is to construct a monomial basis of
$gr V(\la)$. The elements $\prod_{\al>0} f_\al^{s_\al} v_\la$ with $s_\al\ge 0$
obviously span $gr V(\la)$ (recall that the order in  $\prod_{\al>0} f_\al^{s_\al}$
is not important since $f_\al$ are considered as elements of $S(\n^-)$).
For each $\la$ we construct a set $S(\la)$ of multi-exponents
$\bs=\{s_\al\}_{\al>0}$ such that the elements
$$
f^\bs v_\la=\prod_{\al>0} f_\al^{s_\al}v_\la, \ \bs\in S(\la)
$$
form a basis of $gr V(\la)$.
To give a definition of $S(\la)$ we need
the notion of a symplectic version of {\it Dyck path}, which is precisely defined in
Section \ref{secdef}, Definition~\ref{dyckpath}.
The definition is similar to the one for usual Dyck paths,
see for example
\cite{FFoL}). In short, a Dyck path $\bp=(p(0),\dots,p(k))$ is a sequence of positive roots
starting at a simple root $\al_{i}$, ending at a root $\al_{j}$ or
$\al_{j,\ol j}$, $j\ge i$ and
obeying some recursion rules. We denote by $\D$ the set of all Dyck paths.

For a dominant weight $\la$ we introduce a polytope
$P(\lam)\subset \R^{n^2}_{\ge 0}$:
\[
P(\lam):=\bigg\{(s_ \al)_{\al> 0}\mid \forall\bp\in\D:
\begin{array}{l}
\text{ If }p(0)=\al_i,p(k)=\al_j, \text{ then }\\
s_{p(0)} + \dots + s_{p(k)}\le (\la,\al_{i,j}),\\
\text{ if } p(0)=\al_i,p(k)=\al_{j,\ol{j}}, \text{ then }\\
s_{p(0)} + \dots + s_{p(k)}\le (\la,\al_{i,n})\\
\end{array}\bigg\}.
\]
Let $S(\la)$ be the set of integral points in $P(\lam)$.

We show:
\vskip 5pt\noindent
{\bf Theorem~B}.\
{\it The set of elements $f^\bs v_\la$, $\bs\in S(\la)$, forms a basis of $gr V(\la)$.}
\vskip 5pt
For $\bs \in S(\lambda)$ define the weight
$$
\wt(\bs) := \sum_{1 \leq j \leq k \leq n} s_{\al_{j,k}} \alpha_{j,k}
+\sum_{1 \leq j \leq k < n} s_{\al_{j,\ol k}} \alpha_{j,\ol k}.
$$
As an important application we obtain:
\begin{cor}\label{application-cor}
\begin{itemize}
\item[]
\item[i)]
For each $\bs\in S(\la)$ fix an arbitrary order of factors $f_\al$ in the product
$\prod_{\al >0} f_\al^{s_\al}$.
Let $f^\bs=\prod_{\al >0} f_\al^{s_\al}$ be the ordered product.
Then the elements $f^\bs v_\la$, $\bs\in S(\la)$, form a basis of $V(\la)$.
\item[ii)] $\dim V(\lam)=\sharp S(\la)$.
\item[iii)] $char V(\lam) =\sum_{\bs\in S(\lam)} e^{\la-\wt(\bs)}$.
\end{itemize}
\end{cor}
We note that the order in the corollary above is important since we are back to the
action of the (in general) not commutative enveloping algebra.
We thus obtain a family of bases for irreducible $\msp_{2n}$-modules.
The existence of these bases (with the same indexing set)
was proved by Vinberg for ${\mathfrak{sp}}_4$ (see \cite{V}).

The modules $gr V(\la)$ have one more nice property. Namely, given two dominant
integral weights $\la$ and $\mu$, consider the subspace
$gr V(\la,\mu)\hk gr V(\la)\T gr V(\mu)$ generated from the product of highest weight
vectors: $gr V(\la,\mu)=S(\fn^-)(v_\la\T v_\mu)$. We prove that
$gr V(\la,\mu)\simeq gr V(\la+\mu)$ as $\fn^-$-modules. This is an analogue
of the corresponding classical result. In type $A$ this statement was proved
in \cite{FFoL}. Dualizing the embedding $gr V(\la+\mu)\hk gr V(\la)\T gr V(\mu)$, one
obtains an algebra structure on the space $\bigoplus_{\lam} (gr V(\lam))^*$.
The projective spectrum of this algebra is a certain deformation of the symplectic
flag variety. In type $A$ it was studied in \cite{F3}.

\begin{rem}
The data labeling the basis vectors is similar to that for the symplectic Gelfand-Tsetlin
patterns (see \cite{GT}, \cite{BZ}). However, these bases are very different
from the symplectic GT basis.
On the combinatorial side the connection with the Gelfand-Tsetlin patterns was
recently clarified by Ardila, Bliem and Salazar \cite{ABS}.
Generalizing a result of Stanley, they show that for every partition $\lambda$
there exists a marked poset $(P, A, \la)$ such that the Gelfand-Tsetlin polytope
coincides with the corresponding marked order polytope and our polytope
$P(\la)$ coincides with
the corresponding marked chain polytope.
Note that both polytopes have the same Ehrhart polynomials \cite{ABS}.
\end{rem}

We finish the introduction with several remarks. The PBW filtration for
highest weight representations was considered in \cite{FFoL}, \cite{Kum},
\cite{FFJMT}, \cite{F1}, \cite{F2}, \cite{F3}.
It was shown that it has important applications in algebraic geometry,
representation theory of current and affine algebras and in mathematical physics.

There exist special representations $V(\la)$ such that the operators
$f^\bs$ consist only of mutually commuting root vectors, even before passing to
$gr V(\la)$. These modules can be described via the theory of abelian radicals
and turned out to be important in the theory of vertex operator algebras
(see \cite{GG}, \cite{FFL}, \cite{FL}).

Finally we note that $gr V(\la)$ carries an additional grading on each weight space $V(\la)^\mu$
of $V(\la)$:
$$
gr V(\la)^\mu=\bigoplus_{s\ge 0} gr_s V(\la)^\mu=\bigoplus_{s\ge 0} V(\la)^\mu_s/V(\la)^\mu_{s-1}.
$$
The graded character of the weight space is the polynomial
$$
p_{\lam,\mu}(q):=\sum_{s\ge 0}(\dim V(\la)^\mu_s/V(\la)^\mu_{s-1})q^s.
$$
Define the degree
$$\deg(\bs) :=
\sum_{1 \leq j \leq k \leq n} s_{\al_{j,k}}
+\sum_{1 \leq j \leq k < n} s_{\al_{j,\ol k}}
$$
for $\bs\in S(\lam)$, and let
$S(\lam)^\mu$ be the subset of elements such that $\mu=\lam-\wt(\bs)$.
Then
\begin{cor*}
$p_{\lam,\mu}(q)=\sum_{\bs\in S(\lam)^\mu} q^{\deg \bs}$.
\end{cor*}
We note that our filtration is different from the Brylinski-Kostant filtration
(see \cite{Br}, \cite{Kos}).

Our paper is organized as follows:\\
In Section 1 we introduce notations and state the problems.
Sections 2 and 3 are devoted to the proof of Theorems $A$ and $B$.
In Section 2 we prove the spanning property of our basis
and in Section 3 we finalize the proof.

\section{Definitions}\label{secdef}
Let $R^+$ be the set of positive roots of $\msp_{2n}$.
For each $\al\in R^+$ we fix a non-zero element $f_\al\in\fn^-_{-\al}$.
Let
$\al_i$, $\omega_i$ $i=1,\dots,n$ be the simple roots and the fundamental weights.
All positive roots of $\msp_{2n}$
can be divided into two groups:
\begin{gather*}
\al_{i,j}=\al_i+\al_{i+1}+\dots +\al_j,\ 1\le i\le j\le n,\\
\alpha_{i, \ol{j}} = \alpha_i + \alpha_{i+1} + \ldots +
\alpha_n + \alpha_{n-1} + \ldots + \alpha_j, \ 1\le i\le j\le n
\end{gather*}
(note that $\al_{i,n}=\al_{i,\ol n}$).
We will use the following short versions
$$\al_i = \al_i,\ \al_{\ol{i}} = \al_{i,\ol{i}},\ f_{i,j}=f_{\al_{i,j}},\
f_{i,\ol j}=f_{\al_{i,\ol j}}.$$
We recall the usual order on the alphabet $J = \{1, \ldots, n, \ol{n-1}, \ldots, \ol{1}\}$
$$ 1 <2 < \ldots < n-1 < n < \ol{n-1} < \ldots < \ol{1}.$$

Let $\msp_{2n}=\n^+\oplus\h\oplus \n^-$ be the Cartan decomposition.
Consider the increasing degree  filtration on
the universal enveloping algebra of $\U(\n^-)$:
\begin{equation}\label{df}
\U(\n^-)_s=\mathrm{span}\{x_1\dots x_l:\ x_i\in\n^-, l \le s \},
\end{equation}
for example, $\U(\n^-)_0=\C \cdot 1$.

For a dominant integral weight  $\la=m_1\omega_1 + \dots + m_n\omega_n$ let
$V(\la)$ be the corresponding irreducible highest weight
$\msp_{2n}$-module with a highest weight vector $v_\la$.
Since $V(\la)=\U(\n^-)v_\la$, the filtration \eqref{df} induces an increasing  filtration $V(\la)_s$ on $V(\la)$:
\[
V(\la)_s=\U(\n^-)_s v_\la.
\]
We call this filtration the PBW filtration and study the associated graded
space $gr V(\la)$.  In the following lemma we describe some operators acting on
$gr V(\la)$. Let $S(\n^-)$ denotes the symmetric algebra of $\n^-$.
\begin{lem}
The action of  $\U(\n^-)$ on $V(\la)$ induces the structure of a $S(\n^-)$-module on
$gr V(\la)$ and
\[
gr(V(\la))=S(\n^-)v_\la.
\]
The action of $\U(\n^+)$ on $V(\la)$ induces the structure of a $U(\n^+)$-module on
$gr V(\la)$.
\end{lem}

Our aims are:
\begin{itemize}
\item to describe $gr V(\la)$ as an  $S(\n^-)$-module, i.e. describe the ideal
$I(\la)\hk S(\n^-)$ such that $gr V(\la)\simeq S(\n^-)/I(\la)$;
\item to find a basis of $gr V(\la)$.
\end{itemize}
The description of the ideal is given in the introduction (see Theorem $A$).
To describe the basis
we introduce the notion of the symplectic Dyck paths:
\begin{dfn}\label{dyckpath}\rm
A symplectic Dyck path (or simply a path) is a sequence
\[
\bp=(p(0), p(1),\dots, p(k)), \ k\ge 0
\]
of positive roots satisfying the following conditions:
\begin{itemize}
\item[{\it a)}] the first root is simple, $p(0)=\al_i$ for some $ 1 \leq i \leq n$;
\item[{\it b)}] the last root is either simple or the highest root of 
a symplectic subalgebra, more precisely $p(k) = \al_j$ or $p(k) = \al_{\ol{j}}$ for some
$ i \le j \leq n$;
\item[{\it c)}] the elements in between obey the following recursion rule:
If $p(s)=\al_{r,q}$ with $r, q \in J$ then the next element in the sequence is of the form either
$p(s+1)=\al_{r,q+1}$  or $p(s+1)=\al_{r+1,q}.$ where $x+1$ denotes the smallest element in $J$ which is bigger than $x$.
\end{itemize}
\end{dfn}

To give a visual interpretation of the notion of a Dyck-path for $\msp_{8}$,
arrange the positive roots
in the form of a triangle. In this picture, a Dyck path is a path in the directed graph, starting at a simple root root and ending at one of the edges.

$$
\begin{array}[pos]{ccccccccccccc}
\alpha_{1,1} & \rightarrow & \alpha_{1,2} & \rightarrow & \alpha_{1,3}& \rightarrow & \alpha_{1,4}& \rightarrow & \alpha_{1,\ol{3}}& \rightarrow & \alpha_{1,\ol{2}}& \rightarrow & \alpha_{1,\ol{1}}\\
&& \downarrow && \downarrow && \downarrow && \downarrow && \downarrow && \\
&&\alpha_{2,2} & \rightarrow & \alpha_{2,3}& \rightarrow & \alpha_{2,4}& \rightarrow & \alpha_{2,\ol{3}}& \rightarrow & \alpha_{2,\ol{2}} &&\\
&& && \downarrow && \downarrow && \downarrow && && \\
&& &  & \alpha_{3,3}& \rightarrow & \alpha_{3,4}& \rightarrow & \alpha_{3,\ol{3}}&  &  &&\\
&& &&  && \downarrow && && && \\
&& &  & &  & \alpha_{4,4}& & &  &  &&\\
\end{array}
$$

Denote by $\D$ the set of all Dyck-paths.
For a dominant weight $\la=\sum_{i=1}^n m_i\omega_i$
let $P(\lam)\subset \R^{n^2}_{\ge 0}$ be the polytope
\begin{equation}
\label{polytopeequation}
P(\lam):=\bigg\{(s_ \al)_{\al> 0}\mid \forall\bp\in\D:
\begin{array}{l}
\text{ If }p(0)=\al_i,p(k)=\al_j, \text{ then }\\
s_{p(0)} + \dots + s_{p(k)}\le m_i + \dots + m_j,\\
\text{ if } p(0)=\al_i,p(k)=\al_{\ol{j}}, \text{ then }\\
s_{p(0)} + \dots + s_{p(k)}\le m_i + \dots + m_n\\
\end{array}\bigg\},
\end{equation}
and let $S(\la)$ be the set of integral points in $P(\lam)$.

\vskip 5pt\noindent
For a multi-exponent $\bs=\{s_ \beta\}_{\beta>0}$, $s_ \beta\in\Z_{\ge 0}$, let $f^\bs$ be the element
\[
f^\bs=\prod_{\beta\in R^+} f_ \beta^{s_ \beta}\in S(\fn^-).
\]

In the next two sections we prove the following theorem (Theorem $B$ from
the introduction), which immediately implies Corollary~\ref{application-cor}.

\begin{thm}\label{basis}
The set $f^\bs v_\la$, $\bs\in S(\la)$, forms a basis of $gr V(\la)$.
\end{thm}
\proof
In Section 2 we show that the elements $f^\bs v_\la$, $\bs\in S(\la)$, span $gr V(\la)$, see
Theorem~\ref{spanCn}. In Section 3 we show that the elements are linear independent in $gr V(\la)$ (see Theorem~\ref{maintheorem}), which finishes
the proof.
\qed

\section{The spanning property}\label{spanningsection}
We start with writing down the powers of certain positive roots annihilating
a highest weight vector in an irreducible $\msp_{2n}$-module.

\begin{lem}
Let $\la=\sum_{i=1}^n m_i\omega_i$ be the $\msp_{2n}$-weight and let  $V(\la)$
be the corresponding highest weight module with highest weight vector $v_\la$.
Then
\begin{gather}
f_{\al_{i,j}}^{m_i+\dots +m_j+1}v_\la=0,\ 1\le i\le j\le n-1,\label{1}\\
f_{\alpha_{i, \ol{i}}}^{m_i+\dots + m_n+1}v_\la=0,\ 1\le i\le n.\label{2}
\end{gather}
\end{lem}
\begin{proof}
For each positive root $\al$ we have the corresponding $\msl_2$-triple
$\{e_\al$, $h_\al$, $f_\al\}$.
Now the lemma follows immediately from the $\msl_2$-theory.
\end{proof}

In the following we use the differential operators $\pa_\al$ defined by
$$
\pa_\al f_\beta=
\begin{cases}
f_{\beta-\al},\  \text{ if }  \beta-\al\in\triangle^+,\\
0,\ \text{ otherwise}.
\end{cases}
$$
As in the ${\tt A}_n$-case (see \cite{FFoL}),
we have a natural action of $U(\n^+)$ on
$S(\n^-)$ coming from the natural action of $U(\n^+)$ on $S(\g)$
and the identification $S(\n^-)\simeq S(\g)/S(\n^-)S_+(\h\oplus\n^+)$.
The operators $\pa_\al$ satisfy the property
\[
\pa_\al f_\beta = c_{\al,\beta} (\text{ad}\, e_\al)(f_\beta),
\]
where $c_{\al,\beta}$ are some non-zero constants.
In what follows we sometimes use the equality $\al_{i,\ol{n}}=\al_{i,n}$.
\begin{lem}\label{pabeal}
The only non-trivial vectors of the form $\pa_\beta f_\al$, $\al,\beta>0$ are as follows:
for $\al=\al_{i,j}$, $1\le i\le j\le n$
\begin{equation}\label{sp1}
\pa_{i,s}f_{i,j}=f_{s+1,j},\ i\le s<j,\quad
\pa_{s,j}f_{i,j}=f_{i,s-1},\ i< s\le j,
\end{equation}
and for $\al=\al_{i,\ol{j}}$, $1\le i\le j\le n$
\begin{gather}\label{sp2}
\pa_{i,s}f_{i,\ol{j}}=f_{s+1,\ol{j}},\ i\le s <j,\quad
\pa_{i,s}f_{i,\ol{j}}=f_{j,\ol{s+1}},\ j\le s,\quad
\pa_{i,\ol{s}}f_{i,\ol{j}}=f_{j,s-1},\ j<s,\\
\label{sp3}
\pa_{s+1,\ol{j}}f_{i,\ol{j}}=f_{i,s},\ i\le s <j,\quad
\pa_{j,\ol{s+1}}f_{i,\ol{j}}=f_{i,s},\ j\le s,\quad
\pa_{j,s-1}f_{i,\ol{j}}=f_{i,\ol{s}},\ j<s.
\end{gather}
\end{lem}

Let us illustrate this lemma by the following picture in type ${\tt C}_5$.
\vskip 5pt
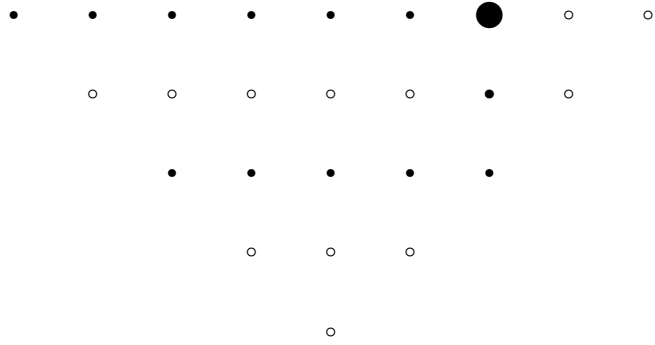
\[
\begin{picture}(200,150)
\multiput(0,150)(30,0){6}{\circle*{3}}
\put(180,150){\circle*{10}}
\multiput(210,150)(30,0){2}{\circle{3}}

\multiput(30,120)(30,0){7}{\circle{3}}
\put(180,120){\circle*{3}}

\multiput(60,90)(30,0){5}{\circle*{3}}

\multiput(90,60)(30,0){3}{\circle{3}}
\multiput(120,30)(30,0){1}{\circle{3}}
\end{picture}
\]

Here all circles correspond to the positive roots of the root system of type ${\tt C}_5$
in the following way:
in the upper row we have from left to right  $\al_{1,1}, \dots, \al_{1,5}, \al_{1,\ol{4}},\dots, \al_{1,\ol{1}}$,
in the second row we have from left to right $\al_{2,2}, \dots, \al_{2,5}, \al_{2,\ol{4}},\dots, \al_{2,\ol{2}}$,
and the last line corresponds to the root $\al_{5,5}$. Now let us take the root
$\al_{1,\ol{3}}$ (which corresponds to the fat circle). Then all roots which
can be obtained by applying the operators $\pa_\beta$ are depicted as filled
small circles.

The following remark will be important for us.
\begin{rem}\label{remAC}
Formula \eqref{sp1} reproduces the picture in type $A_n$.
Formulas \eqref{sp1}, \eqref{sp2} and \eqref{sp3} resemble the
situation in type $A_{2n-1}$. The difference is that in the symplectic case
the roots $\pa_\beta f_\al$ with fixed $\al$ do not form two segments
(as in type $A$), but three segments.
\end{rem}

Our goal is to prove the following theorem.
\begin{thm}\label{spanCn}
\begin{itemize}
\item[{\it i)}] The vectors $f^\bs v_\la$, $\bs\in S(\la)$ span $gr V(\la)$.
\item[{\it ii)}] Let $I(\la)$ be the ideal $I(\la)=S(\n^-)(\U(\n^+)\circ R)$, where
$$
R=\mathrm{span}\{ f_{\al_{i,j}}^{m_i+\dots +m_j+1}, 1\le i\le j\le n-1,\
f_{\alpha_{i, \ol{i}}}^{m_i+\dots + m_n+1}, 1\le i\le n\}.
$$
There exists a monomial order on
$S(\n^-)=\bc[f_\alpha\mid\alpha>0]$, denoted by  `` $\ord$", such that
for any $\bs\not\in S(\la)$ there exists a homogeneous expression
(a straightening law) of the form
\begin{equation}\label{straighteninglaw}
f^\bs -\sum_{\bs\ord \bt} c_\bt f^\bt \in I(\la).
\end{equation}
\end{itemize}
\end{thm}

\begin{rem}
In the following we refer to $(\ref{straighteninglaw})$ as a {\it straightening law} for
the polynomial ring $S(\n^-)=\bc[f_\alpha\mid\alpha>0]$ with respect to the ideal $I(\la)$.
Such a straightening law implies that in the quotient ring $S(\n^-)/I(\la)$ we can express
$f^\bs$ for $\bs\not\in S(\la)$ as a linear combination of monomials which are
smaller in the monomial order than $f^\bs$, but of the same total degree since the
expression in $(\ref{straighteninglaw})$ is homogeneous.
\end{rem}

We show first that {\it ii)} implies {\it i)}:
\begin{proof} {\bf [{\it ii)} $\Rightarrow ${\it i)}]} The elements in $R$ obviously
annihilate $v_\la\in gr V(\la)$, and so do the elements of $\U(\n^+)\circ R$,
and hence so do the
elements of the ideal $I(\la)$. As a consequence we get a surjective map $S(\n^-)/I(\la)\rightarrow gr V(\la)$.

Suppose $\bs\not\in S(\la)$. We know by {\it ii)} that
$f^\bs = \sum_{\bs\ord \bt} c_\bt f^\bt$
in $S(\n^-)/I(\la)$. If some $\bt$ with nonzero coefficient $c_\bt$ is not an element of $S(\la)$,
then we can again apply a straightening law and replace $f^\bt$ by a linear combination of
smaller monomials. Since there are only finite number of monomials
of the same total degree, by repeating the procedure if necessary, after a finite number of steps
we obtain an expression of $f^\bs$ in $S(\n^-)/I(\la)$ as a linear combination of elements $f^\bt$, $\bt\in S(\la)$.
It follows that the set $\{ f^\bt\mid \bt\in S(\la)\}$ is a spanning set for $S(\n^-)/I(\la)$, and hence, by the
surjection above, we get a spanning set $\{ f^\bt v_\la\mid \bt\in S(\la)\}$ for $gr V(\la)$.
\end{proof}

To prove the second part we need to define the total order. We start by defining a total
order on the variables:
\begin{equation}\label{order}
\begin{array}{rcl}
&f_{n,n}&>\\
f_{n-1,\overline{n-1}}>&f_{n-1, n}&>f_{n-1, n-1}>\\
f_{n-2, \overline{n-2}}>f_{n-2, \overline{n-1}}>&f_{n-2, n}&>f_{n-2, n-1}>f_{n-2, n-2}>\\
\ldots>&\ldots&>\ldots >\\
f_{1, \overline{1}}>f_{1, \overline{2}}>\ldots >f_{1, \overline{n-1}}>&f_{1,n}&>f_{1,n-1}>\ldots>f_{1,2}>f_{1,1}.
\end{array}
\end{equation}
We use the same notation for the induced homogeneous lexicographic ordering on the monomials.
Note that this monomial order $>$ is not the order $\ord$. To define the latter, we need some more notation.
Let
\begin{gather*}
s_{\bullet, j}=\sum_{i=1}^j s_{i,j},\quad s_{\bullet, \ol{j}}=\sum_{i=1}^j s_{i,\ol{j}}, \\
s_{i,\bullet}=\sum_{j=i}^n s_{i,j} + \sum_{j=i}^{n-1} s_{i,\ol{j}}.
\end{gather*}
Define a map $d$ from the set of multi-exponents $\bs$ to $\Z_{\ge 0}^n$:
\[
d(\bs)=(s_{n,\bullet},s_{n-1,\bullet},\dots,s_{1,\bullet}).
\]
So, $d(\bs)_i=s_{n-i+1,\bullet}$.
We say $d(\bs)>d(\bt)$ if there exists an $i$ such that
$$
d(\bs)_1=d(\bt)_1, \dots, d(\bs)_i=d(\bt)_i, d(\bs)_{i+1} > d(\bt)_{i+1}.
$$
\begin{dfn}\label{sp<}
For two monomials $f^\bs$ and $f^\bt$  we say $f^\bs \ord f^\bt$ if either
\begin{itemize}
\item[{\it a)}] the total degree of $f^\bs$ is greater than the total degree of $f^\bt$;
\item[or {\it b)}] both have the same total degree, but $d(\bs) < d(\bt)$;
\item[or {\it c)}] both have the same total degree, $d(\bs) = d(\bt)$, but $f^\bs > f^\bt$.
\end{itemize}
\end{dfn}
In words, if both have the same total degree, this definition says that $f^\bs$ is greater than $f^\bt$ if
$d(\bs)$ is smaller than $d(\bt)$,  or $d(\bs)=d(\bt)$ but $f^\bs > f^\bt$
with respect to the homogeneous lexicographic ordering on $\bc[f_\alpha\mid \alpha >0]$.

\begin{rem}
It is easy to check that ``$\ord$" defines a {\it monomial ordering}, i.e.,
if $f^\bs\ord f^\bt$ and $f^\bm\not=1$, then
$$
f^{\bs}f^{\bm}=f^{\bs+\bm}\ord  f^{\bt}f^{\bm}=f^{\bt+\bm}\ord f^\bt.
$$
\end{rem}

Slightly abusing notations, we use the same symbol $\ord$  also for the multi-exponents: we write
$\bs \ord \bt$ if and only if $f^\bs \ord f^\bt$.

\vskip 3pt\noindent
{\it Proof of Theorem  \ref{spanCn} {\it ii)}}.
We discuss first some reduction steps. Let $\bs$ be a multi-exponent violating
some of the Dyck paths condition from the definition of $S(\la)$ and let $\bp$
be a corresponding Dyck path. We write $\bs$ as a sum $\bs=\bs^1+\bs^2$,
where $\bs^1$ is defined as follows: $s_{\al}^1=s_{\al}$ if $\alpha\in\bp$
and $s_{\al}^1=0$ if $\alpha\not\in\bp$, so $s_{\al}^1$ has support (i.e.
nonzero entries) only on $\bp$. Now obviously we still have $\bs^1\not\in S(\la)$.
If we have a straightening law for $f^{\bs^1}$:
$$
f^{\bs^1} - \sum_{\bs^1\ord \bt} c_\bt f^\bt \in I(\la)
$$
then multiplication by $f^{\bs^2}$ gives
a straightening law for $f^{\bs}=f^{\bs^1}f^{\bs^2}$, because $\ord$ is a monomial
order.

So it suffices to find a straightening law for those $\bs\not\in S(\la)$
which are supported on a Dyck path $\bp$ and $\bs$ violates the Dyck path condition for
$S(\la)$ for the path $\bp$.

Suppose first that the Dyck path $\bp$ is such that $p(0)=\al_i$, $p(k)=\al_j$ for some $1\le i\le j< n$.
We are going to show that in this case we get a straightening
law by the corresponding result
for the Lie algebra $\mathfrak{sl}_n$ from \cite{FFoL}.
In fact, consider the Lie subalgebra $M\subset  \mathfrak{sp}_{2n}$
generated by the elements $e_{\alpha_{i,i}}, f_{\alpha_{i,i}}, h_{\alpha_{i,i}}$, $1\le i <n$. This subalgebra
is isomorphic to  $\mathfrak{sl}_n$. Let $M=\n^+_M\oplus\h_M\oplus \n^-_M$ be the Cartan decomposition
obtained by setting $\n^+_M=\n^+\cap M$, $\n^-_M=\n^-\cap M$ and $\h_M=\h\cap M$. Let
$$
R_M=\mathrm{span}\{ f_{\al_{i,j}}^{m_i+\dots +m_j+1}, 1\le i\le j\le n-1\}\subset S(\n^-_M)\subset S(\n^-).
$$
Then $R_M\subset R$ and $U(\n^+_M)\circ R_M\subset U(\n^+)\circ R$. Set $\la_M=\sum_{i=1}^{n-1} m_i\om_i$ and
let $I_M(\la_M)$ be the ideal $I_M(\la):= S(\n^-_M)(U(\n^+_M)\circ R_M)\subset S(\n^-_M)$.

We have an obvious inclusion $I_M(\la_M)\subset I(\la)$. But note that the ideal $I_M(\la_M)$ is 
considered in \cite{FFoL} for the $\mathfrak{sl}_n$-case
(recall, $M\simeq \mathfrak{sl}_n$).
Also the Dyck path considered here is an $\msl_n$ Dyck path,
because all roots occurring in the path are roots in the subroot-system
associated to the Lie-subalgebra $M$.
It follows by \cite{FFoL} that we have a straightening law
\begin{equation}\label{slstraight}
f^\bs -\sum_{\bs>\bt} c_\bt f^\bt \in I_M(\la_M)\subset I(\la).
\end{equation}
It remains to show that in the sum above we can replace ``$>$" by ``$\ord$".
For this we need to recall the proof in the type ${\tt A}$-case. Recall that we work now with the
subalgebra
$M\simeq{\mathfrak{sl}}_n\subset {\mathfrak{sp}}_{2n}$. To get the straightening law above,
one starts with the element $f_{1,i}^{s_{p(0)}+\dots + s_{p(k)}}\in R_M$. 
Applying the $\pa$-operators
(see \cite{FFoL}) one shows that
\[
B=f_{1,1}^{s_{\bullet, 1}}f_{1,2}^{s_{\bullet ,2}}\dots f_{1,i}^{s_{\bullet, i}} \in R_M.
\]
One applies then the following $\pa$-operators to $B$ to get
\begin{equation}\label{Aoperator}
A=\pa_{1,1}^{s_{2,\bullet}}\pa_{1,2}^{s_{3,\bullet}}\dots \pa_{1,i-1}^{s_{i,\bullet}} B\in R_M
\end{equation}
(since $\bs$ is supported on $\bp$ and $p(k)=\al_j$, $j<n$, we have $s_{l,\bullet}=\sum_{j=l}^{n-1} s_{l,j}$).
We show in \cite{FFoL} that
\begin{equation}
A =\sum_{\bt\le \bs} c_\bt F^\bt
\end{equation}
for some $c_\bs\ne 0$,
which gives rise to the straightening law in $(\ref{slstraight})$.
Now in this special case Lemma~\ref{pabeal} implies that the application of the $\pa$-operators
in $(\ref{Aoperator})$ produces only summands such that $d(\bs)=d(\bt)$
for any $\bt$ occurring in the sum with
a nonzero coefficient. Hence we can replace ``$>$" by ``$\ord$" in $(\ref{slstraight})$, which finishes the proof
of the theorem in this case.

Now assume $p(0)=\al_{i,i}$ and $p(k)=\al_{j,\ol{j}}$ for some $j\ge i$.
We include the case $j=n$ by writing
$\al_{n,n}=\al_{n,\ol{n}}$. We proceed by induction on $n$.
For $n=1$ we have ${\mathfrak{sp}}_{2}={\mathfrak{sl}}_{2}$, so we can refer
to \cite{FFoL}.
Now assume we have proved the existence of a straightening law for all symplectic algebras
of rank strictly smaller than $n$. If $i>1$, then the Dyck path is also a Dyck path
for the symplectic subalgebra
$L\simeq \mathfrak{sp}_{2n-2(i-1)}$ generated by  $e_{\alpha_{k,k}}, f_{\alpha_{k,k}}, h_{\alpha_{k,k}}$, $i\le k \le n$.
Let $\n^+_L, \n^-_L$ etc. be defined by the intersection of $\n^+,\n^-$ etc. with $L$ and set $\la_L=\sum_{k=i}^n m_k\om_k$.
It is now easy to see that the straightening law for $f^\bs$ viewed as an element in $S(\n^-_L)$
with respect to $I_L(\la_L)$ defines also a straightening law for $f^\bs$ viewed as an element in
$S(\n^-)$ with respect to $I(\la)$.

So from now on we fix $p(0)=\al_1$ and $p(k)=\al_{i,\ol{i}}$ for some $i\in\{1,\dots,n\}$.
For a multi-exponent $\bs$ supported on $\bp$, set
$$
\Sigma=\sum_{l=0}^k s_{p(l)} > m_1+\dots +m_n.
$$
We have obviously $f_{1,\bar{1}}^{\Sigma}\in I(\lam)$. We consider now two
operators
$$
\Delta_1:=\pa_{1,i-1}^{s_{\bullet,\bar i}+s_{i,\bullet}}
\underbrace{
\pa_{i+1,\ol{i+1}}^{s_{\bullet,i}}\ldots\pa_{n,\bar n}^{s_{\bullet,n-1}}
}_{\delta_3}
\underbrace{
\pa_{1,n-1}^{s_{\bullet,n-1}+s_{\bullet,\ol{n}}}
\ldots
\pa_{1,i}^{s_{\bullet,i}+s_{\bullet,\ol{i+1}}}
}_{\delta_2}
\underbrace{
\pa_{1,\bar i}^{s_{\bullet,i-1}}\dots \pa_{1,\bar 3}^{s_{\bullet,2}}\pa_{1,\bar 2}^{s_{\bullet,1}}
}_{\delta_1}
$$
(so $\Delta_1:=\pa_{1,i-1}^{s_{\bullet,\bar i}+s_{i,\bullet}}\delta_3\delta_2\delta_1$)
and
$$
\Delta_2:=\pa_{1,1}^{s_{2,\bullet}}\pa_{1,2}^{s_{3,\bullet}}\dots \pa_{1,i-2}^{s_{i-1,\bullet}}.
$$
We will show that
\begin{equation}\label{goalstraightening}
\Delta_2\Delta_1 f_{1,\bar{1}}^{\Sigma}=c_\bs f^\bs + \sum_{\bs\ord\bt} c_\bt f^\bt
\end{equation}
with complex coefficients $c_\bs, c_\bt$, where $c_\bs\not=0$. Since
$\Delta_2\Delta_1 f_{1,\bar{1}}^{\Sigma}\in I(\lam)$,
the proof of $(\ref{goalstraightening})$ finishes the proof of the theorem.
A first step in the proof of $(\ref{goalstraightening})$ is the following lemma.

Recall the alphabet $J = \{1, \ldots, n, \ol{n-1}, \ldots, \ol{1}\}$.
Let $q_1,\dots,q_i\in J$ be a sequence of increasing elements defined by
\[
q_k=\max\{l\in J:\ \al_{k,l}\in\bp\}.
\]
For example, $q_i=\ol i$. The roots of $\bp$ are then of the form
\[
\al_{1,1},\dots, \al_{1,q_1},\al_{2,q_1},\dots,\al_{2,q_2},\dots,
\al_{i,q_{i-1}},\dots,\al_{i,q_i}.
\]

\begin{lem}\label{B}
Set $f^{\bs'}=f_{1,1}^{s_{\bullet,1}}f_{1,2}^{s_{\bullet,2}} \ldots
f_{1,q_{i-1}}^{s_{\bullet,q_{i-1}}-s_{i,q_{i-1}}} f_{i,q_{i-1}}^{s_{i,q_{i-1}}}
\ldots f_{i,\bar i}^{s_{i,\bar i}}$.  Then $\Delta_1 f_{1,\bar{1}}^{\Sigma}$
is of the form
\begin{equation}\label{equationin2.8}
\Delta_1 f_{1,\bar{1}}^{\Sigma}=c_{\bs'} f^{\bs'} + \sum_{\bs'\ord\bt} c_\bt f^\bt
\end{equation}
such that $c_{\bs'}\not=0$. In addition, if $f^\bt $, $\bt\not=\bs'$,
is a monomial occurring in this sum, then one of the following statements holds:
\begin{itemize}
\item there exists an index $j$ such that
$d(\bt)_j>0$ for some $j\in \{1,2,...,n-~i\}$,
\item $d(\bt)_j=0$ for all $j\in \{1,2,...,n-i\}$ and $d(\bt)_{n-i+1}>s_{i,\bullet}$,
\item $d(\bt)=d(\bs')$ and
$f_{i,i}^{t_{i,i}}f_{i,i+1}^{t_{i,i+1}}\cdots f_{i,\bar i}^{t_{i,\bar i}}<
f_{i,i}^{s_{i,i}}f_{i,i+1}^{s_{i,i+1}}\cdots f_{i,\bar i}^{s_{i,\bar i}}$
\end{itemize}
\end{lem}
Before proving the lemma, we explain in the following corollary the reason why we need the lemma.
The corollary is proved after the proof of the lemma.
\begin{cor}\label{unimportantmonomial}
If $f^\bt\not= f^{\bs'}$ is a monomial occurring in $(\ref{equationin2.8})$,
then $\Delta_2 f^\bt$
is a sum of monomials $f^\bk$ such that $f^\bs \ord f^\bk $.
\end{cor}
\vskip 4pt\noindent
{\it Proof of the lemma.\/}
One sees easily by induction that
$$
\delta_1(f_{1,\bar{1}}^{\Sigma})=f_{1,1}^{s_{\bullet,1}}f_{1,2}^{s_{\bullet,2}}
\ldots f_{1,i-1}^{s_{\bullet,i-1}}f_{1,\bar{1}}^{\Sigma-s_{\bullet,1}-s_{\bullet,2}-\ldots-s_{\bullet,i-1}}.
$$
Since $\al_{1,j}-\al_{1,\ell}$, $1\le j<i$, $i < \ell \le n$, and $\al_{1,j}-\al_{\ell,\bar\ell}$,
$1\le j<i$, $i < \ell \le n$, and $\al_{1,j}-\al_{1,i-1}$, $1\le j<i$,
are never positive roots,
all factors of $\delta_2$ and $\delta_3$, as well as $\pa_{1,i-1}$, annihilate the vector
\[
f^\bx=f_{1,1}^{s_{\bullet,1}}f_{1,2}^{s_{\bullet,2}}\ldots f_{1,i-1}^{s_{\bullet,i-1}}.
\]
Therefore
$$
\Delta_1(f_{1,\bar{1}}^{\Sigma})=
f^\bx\pa_{1,i-1}^{s_{\bullet,\bar i}+s_{i, \bullet}}
\delta_3\delta_2(f_{1,\bar{1}}^{\Sigma-s_{\bullet,1}-s_{\bullet,2}-\ldots-s_{\bullet,i-1}}).
$$

To visualize the following procedure, one should think
of the variables $f_{i,j}$ as being arranged in a triangle like in the picture after Lemma~\ref{pabeal},
or in the following example (type ${\tt C}_4$):
\begin{equation}\label{scheme}
\Skew(0: f_{11},f_{12} ,f_{13} ,f_{14} ,f_{1\bar{3}},f_{1\bar{2}},f_{1\bar{1}} |
1:f_{22} ,f_{23} ,f_{24} ,f_{2\bar{3}},f_{2\bar{2}} |2:f_{33} ,f_{34} ,f_{3\bar{3}}|3:f_{44} )
\end{equation}
With respect to the ordering ``$>$", the largest element is in the bottom row and the
smallest element is in the top row on the left side. We enumerate the rows and
columns like the indices of the variables, so the top row is the 1-st row, the bottom
row the $n$-th row, the columns are enumerated from  left to right,
so we have the $1$-st column on the left side and the most right one is the $\bar{1}$-st column.

The operator $\pa_{1,q}$, $1\le q\le n-1$, kills all $f_{1,j}$ for $1\le j\le q$,
$\pa_{1,q}(f_{1,j})=f_{q+1,j}$ for $j=q+1,\ldots,\ol{q+1}$,
$\pa_{1,q}(f_{1,\bar j})=f_{j, \ol{q+1}}$ for $j=1,\ldots, q$, and $\delta_{1,q}$
kills all $f_{k,\ell}$ for $k\ge 2$. Because of the set of indices of the
operators occurring in $\delta_2$, the operator applied to 
$f_{1,\bar{1}}^{\Sigma-s_{\bullet,1}-s_{\bullet,2}-\ldots-s_{\bullet,i-1}}$ never increases the zero 
entries in the first row, column $\bar i$ up to column $\bar 2$.
As a consequence, the application of $\delta_2$
produces  the monomial
$$
f^\bx f_{1,\ol{i+1}}^{s_{\bullet,i}+s_{\bullet,\ol{i+1}}}\cdots f_{1,\ol{n-1}}^{s_{\bullet,n-2}+s_{\bullet,\ol{n-1}}}
f_{1,n}^{s_{\bullet,n-1}+s_{\bullet,n}} f_{1,\bar{1}}^{s_{\bullet,\bar i}}+\sum c_\bk f^\bk,
$$
where the monomials $f^\bk$ occurring in the sum are such that the corresponding
triangle (see $(\ref{scheme})$) has at least one
non-zero entry in one of the rows between the $(i+1)$-th row and the $n$-th row
(counted from top to buttom).
This implies $d(\bk)_j>0$ for some $j=1,\ldots, n-i$. The operators
$\delta_3$ and $\pa_{1,i-1}^{s_{\bullet,\bar i}+s_{i, \bullet}}$ do not change this property,
because (in the language of the scheme $(\ref{scheme})$ above) the
operators $\pa_{j,\bar j}$ used to compose $\delta_3$ either kill a monomial or, in the
language of the scheme $(\ref{scheme})$, they subtract from an entry in the $\bar j$-th column, $k$-th row
and add to the entry in the same row, but $(j-1)$-th column.
The operator $\pa_{1,i-1}$
subtracts from the entries in the top row. Since the
entries in the top row, column $\ol{i-1}$ up to $\bar 2$, are zero,
it adds to the entries in the $i$-th row. The only exception is $\pa_{1,i-1}$
applied to $f_{1,\bar 1}$, the result is $f_{1,\bar i}$.
It follows that the monomials $f^{\bk'}$
occurring in $\pa_{1,i-1}^{s_{\bullet,\bar i}+s_{i, \bullet}}\delta_3 f^\bk$
have already the desired properties, because we have just seen that
$d(\bk')_j>0$ for some $j=1,\ldots, n-i$.

So to finish the proof of the lemma, it suffices to look at
\begin{gather}\label{whoknows}
f^\bx \pa_{1,i-1}^{s_{\bullet,\bar i}+s_{i, \bullet}} \delta_3 f_{1,\ol{i+1}}^{s_{\bullet,i}+s_{\bullet,\ol{i+1}}}
\cdots f_{1,\ol{n-1}}^{s_{\bullet,n-2}+s_{\bullet,\ol{n-1}}}
f_{1,n}^{s_{\bullet,n-1}+s_{\bullet,n}} f_{1,\bar 1}^{s_{\bullet,\bar i}} \\ =
\label{2.15}
f^\bx \pa_{1,i-1}^{s_{\bullet,\bar i}+s_{i, \bullet}}  f_{1,i}^{s_{\bullet,i}}f_{1,i+1}^{s_{\bullet,i+1}}
\cdots f_{1,n}^{s_{\bullet,n}}
f_{1,\ol{n-1}}^{s_{\bullet,\ol{n-1}}}\cdots
 f_{1,\ol{i+1}}^{s_{\bullet,\ol{i+1}}}f_{1,\bar 1}^{s_{\bullet,\bar i}}
\end{gather}
Note that the operator $\pa_{1,i-1}$ being applied to any variable in
\eqref{2.15} but to $f_{1,\bar 1}$, increases the
degree with respect to the variables $f_{i,*}$ or gives zero.
We note also that $\pa_{1,i-1} f_{1,\ol 1}=f_{1,\bar i}$. So $(\ref{whoknows})$
written as a linear combination $\sum c_\bk f^\bk$ of monomials
such that $d(\bk)_{j}=0$ for $j=1,\ldots,n-i$ and $d(\bk)_{n-i+1}\ge s_{i, \bullet}$.

It remains to consider the case where
$d(\bk)_{n-i+1}= s_{i, \bullet}$. This is only possible
if $\pa_{1,i-1}$ is applied
$s_{\bullet,\bar i}+s_{i,\bullet}$-times to $f_{1,1}^{s_{\bullet,\bar i}}$,
in which case $d(\bk)$ has only two non-zero entries: $d(\bk)_n=\Sigma-s_{i, \bullet}$
and $d(\bk)_{n-i+1}=s_{i, \bullet}$, so $d(\bk)=d(\bs')$. If $\bk\not=\bs'$,
then necessarily
$f_{i,i}^{t_{i,i}}f_{i,i+1}^{t_{i,i+1}}\cdots f_{i,\bar i}^{t_{i,\bar i}}<
f_{i,i}^{s_{i,i}}f_{i,i+1}^{s_{i,i+1}}\cdots f_{i,\bar i}^{s_{i,\bar i}}$.
\qed

\vskip 4pt\noindent
{\it Proof of the corollary.\/}
The operators used to compose $\Delta_2$ do not change anymore the entries
of $d(\bt)$ for the first $n-i+1$ indices.

Suppose first $\bt$ is such that there exists an index $j$ such that
$d(\bt)_j>0$ for some $j\in \{1,2,...,n-i\}$ or $d(\bt)_{n-i+1}>s_{i,\bullet}$.
By the description of the operators
occurring in $\Delta_2$, every monomial $f^\bk$ occurring with a nonzero
coefficient in $\Delta_2 f^\bt$ has this property too and hence $f^\bs \ord f^\bk $.

Next assume $d(\bt)=d(\bs')$ and
$f_{i,i}^{t_{i,i}}f_{i,i+1}^{t_{i,i+1}}\cdots f_{i,\bar i}^{t_{i,\bar i}}<
f_{i,i}^{s_{i,i}}f_{i,i+1}^{s_{i,i+1}}\cdots f_{i,\bar i}^{s_{i,\bar i}}$.
Recall that $\bt_{1,\ol{i-1}}=\ldots=\bt_{1,\ol{1}}=0$. It follows that the operators
occurring in $\Delta_2$ always only subtract from one of the entries in the top row
and add to the entry in the same column and a corresponding row (of index strictly
smaller than $i$). It follows
that all monomials $f^\bk$ occurring in $\Delta_2(f^\bt)$ have the property:
$d(\bk)=d(\bs)$. Since $f_{i,i}^{t_{i,i}}f_{i,i+1}^{t_{i,i+1}}\cdots f_{i,\bar i}^{t_{i,\bar i}}<
f_{i,i}^{s_{i,i}}f_{i,i+1}^{s_{i,i+1}}\cdots f_{i,\bar i}^{s_{i,\bar i}}$,
it follows that $f^\bs> f^\bk$ and hence $f^\bs\ord f^\bk$.
\qed
\vskip 4pt\noindent
{\it Continuation of the proof of Theorem~\ref{spanCn} ii).\/}
We have seen that to prove Theorem~\ref{spanCn} {\it ii)}, it suffices to prove
$(\ref{goalstraightening})$.
By Lemma~\ref{B} and Corollary~\ref{unimportantmonomial}, it remains to prove for
$f^{\bs'}$ that
$\Delta_2 f^{\bs'}$ is a linear combination of  $f^\bs$ with a non trivial coefficient
and monomials strictly smaller than $f^\bs$. The following lemma
proves this claim and hence finishes the proof of the theorem.
\qed

\vskip 4pt
The following lemma completes the proof of  part {\it ii)} of Theorem~\ref{spanCn}.
\begin{lem}\label{straightlemma}
The operator $\Delta_2:=\pa_{1,1}^{s_{2,\bullet}}\pa_{1,2}^{s_{3,\bullet}}\dots
\pa_{1,i-2}^{s_{i-1,\bullet}}$ applied to the monomial
$f^{\bs'}$ is a linear combination of $f^\bs$ and smaller
monomials:
\begin{equation}\label{BC}
\Delta_2 f^{\bs'} =c f^\bs +\sum_{\bs \ord \bt} c_\bt f^\bt, \quad\text{where $c\not=0$}.
\end{equation}
\end{lem}
\begin{proof}
First note that all monomials $f^\bk$ occurring in
$\Delta_2 f^{\bs'}$ have the same total degree.
Recall that $\bs'_{1,\ol{i-1}}=\ldots=\bs'_{1,\ol{1}}=0$. It follows that the operators
occurring in $\Delta_2$ always only subtract from one of the entries in the top row
and add to the entry in the same column and a corresponding row (of index strictly
smaller than $i$ and strictly greater than $1$). Thus all monomials $f^\bk$ occurring
in $\Delta_2(f^{\bs'})$ have the same multi-degree.
In fact, we will see below that $f^\bs$ is a summand and hence $d(\bk)=d(\bs)$.

So in the following we can replace the ordering $\ord$ by $>$ since, in this special case,
the latter implies the first.

The elements $f_{i,j}$ and $f_{i,\bar{j}}$, $2\le i\le j \le n$,
are in the kernel of the operators $\pa_{1,k}$ for all $1\le k\le n$, and
so are the variables $f_{1,j}$, $j\le k$ in the first $k$ columns.

The operator $\pa_{1,k}$, $1\le k\le n$, ``moves" the variables $f_{1,j}$, $k+1\le j\le n$ from the first row to the variable $f_{k+1,j}$ in the
same column.

The operator $\pa_{1,k}$, $1\le k\le n$ ``moves" the variables $f_{1,\bar{j}}$, $k+1\le  j\le n$ from the
first row to the variable $f_{k+1,\bar{j}}$ in the
same column. For $j\le k$, the operator makes the variables switch also the column,
it moves the variable $f_{1,\bar{j}}$ to the variable $f_{j, \overline{k+1}}$ in the $j$-th row and $(\overline{k+1})$-th column.

If $i=1,2$, then $\Delta_2$ is the identity operator,
$f^\bs=f^{\bs'}$ and hence the lemma is trivially true.
Now assume $i\ge 3$.
We note that the monomial
\begin{gather*}
f_{1,1}^{s_{1,1}}\dots f_{1,q_1}^{s_{1,q_1}}\cdot
(\pa_{1,1}^{s_{2,q_1}}f_{1,q_1}^{s_{2,q_1}}\dots
\pa_{1,1}^{s_{2,q_2}}f_{1,q_2}^{s_{2,q_2}})
\cdot \ldots \hskip 80pt \\
\ldots \cdot(\pa_{1,i-2}^{s_{i -1,q_{i-2}}}f_{1,q_{i-2}}^{s_{i-1,q_{i-2}}}\dots
\pa_{1,i-2}^{s_{i-1,q_{i-1}}}f_{1,q_{i-1}}^{s_{i-1,q_{i-1}}})
(f_{i,q_{i-1}}^{s_{i,q_{i-1}}}
\ldots f_{i,\bar i}^{s_{i,\bar i}})
\end{gather*}
is proportional to $f^\bs$ and appears as a summand in $\Delta_2 f^{\bs'}$.
Our goal is to show that all other monomials in $\Delta_2 f^{\bs'}$ are less than $f^\bs$.

All monomials share the common factor $(f_{i,q_{i-1}}^{s_{i,q_{i-1}}}
\ldots f_{i,\bar i}^{s_{i,\bar i}})$. The maximal variable smaller that the ones occurring
in this factor is the variable $f_{i-1,q_{i-1}}$.
Note that if $j<i-1$ then for any $q\in J$ the variable $\pa_{1,j}f_{1,q}$
lies in the $(j+1)$-th row and  $j+1<i$. The operator $\pa_{1,i-2}$
is applied $s_{i-1,\bullet}$ times and the unique maximal monomial
in the sum expression of $\pa_{1,i-2}^{s_{i-1,\bullet}}f^{\bs'}$ is
$$
f_{1,1}^{s_{\bullet,1}}f_{1,2}^{s_{\bullet,2}} \ldots
f_{1,q_{i-2}}^{s_{\bullet,q_{i-2}}- s_{i-1,q_{i-2}}}
(f_{i-1,q_{i-2}}^{s_{i-1,q_{i-2}}} \ldots f_{i-1,q_{i-1}}^{s_{i-1,q_{i-1}}})
(f_{i,q_{i-1}}^{s_{i,q_{i-1}}} \ldots f_{i,\bar i}^{s_{i,\bar i}}).
$$
In fact, applying the operator $\pa_{1,i-2}$ to any of the variables $f_{1,j}$
such that $j\not=q_{i-2},\ldots, q_{i-1}$, one gets a monomial smaller in the order $>$.
The exponents $s_{i-1,j}$, $j=q_{i-2},\ldots,q_{i-1}$, are the maximal powers
such that $\pa_{1,i-2}$ can be applied to $f_{1,j}^{y}$ because either
$q_{i-2}<j<q_{i-1}$, and then $y=s_{\bullet,j}=s_{i-1,j}$, or $j=q_{i-1}$,
then  $s_{i-1,q_{i-1}}$ is the power with which the variable occurs in $f^{\bs'}$,
or $j=q_{i-2}$, then only the power $s_{i-1,q_{i-2}}$ of the operator is left.

Repeating the arguments for the operators  $\pa_{1,i-3}$ etc. we complete the proof
of the lemma.
\end{proof}

\section{Main Theorem}
Recall that in \cite{ABS} the equality $\# S(\la)=\dim V(\la)$
is proved using purely combinatorial tools. Combining this
result with Theorem  \ref{spanCn} we obtain Theorems $A$ and $B$ from the introduction.
However in this section we present  a representation theoretical proof of the equality
$\# S(\la)=\dim V(\la)$ by showing that the vectors $f^\bs$, $\bs\in S(\la)$, are
linearly independent in $gr V(\la)$. The advantage of our proof is that in the course
of the proof we obtain the following important statement: the subspace of
$gr V(\la)\T gr V(\mu)$  generated from the product of highest weight vectors is
isomorphic to $gr V(\la+\mu)$ (see Theorem \ref{maintheorem}, $(iii)$).

\subsection{Fundamental weights and minimal sets}
In this subsection we study the case $\la=\om_i$.
The following lemma follows from the definition of $S(\om_i)$.
\begin{lem}\label{minimalset}
$S(\om_i)$ consists of  all $\bs$ such that $s_\al\le 1$
and the support of $\bs$
is given by the set
$$M^{\bs} = \{  \alpha_{j_l,\ol{k_l}} \; | \; l=1, \ldots , p \}
\cup \{ \alpha_{t_l,r_l} \; | \; l = 1, \ldots, q \}$$
with the following conditions
\begin{gather*}
1 \leq j_1 < j_2 < \ldots < j_p \leq i \; ; \; 1 \leq  k_1 < k_2 < \ldots < k_p,\\
j_p < t_1 < t_2 \ldots < t_q \leq i \leq r_1 < \ldots < r_q \leq n.
\end{gather*}
\end{lem}

\begin{rem}\label{one}
We note that $\sharp M^{\bs} \leq i$ and
every path contains at most one element from $M^{\bs}$,
since the roots on a path are ordered with respect to the order $>$.
\end{rem}

\begin{lem}
For every fundamental weight $\om_i$ we have
$$\sharp S(\om_i) = \dim V(\om_i)$$
\end{lem}
\begin{proof}
Follows from \cite{ABS} or by establishing a bijection with
Kashiwara-Nakashima tableaux or by showing that
\[
\sharp S(\om_i)+ \sharp S(\om_{i-2})+\dots =\binom{2n}{i}
\]
(compare with $\Lambda^i V(\om_1)=V(\om_i)\oplus V(\om_{i-2})\oplus\dots$, see \cite{FH}).
\end{proof}

We set
$$R_i = \{ \beta \in R^+ \; | \; (\om_i, \beta) \neq 0 \}.$$
Let $\lambda = \sum m_j \om_j \in P^+$ and $\bs \in S(\lam)$. We set
$$R_i^{\bs} = \{ \beta \in R_i \; | \; s_{\beta} \neq 0 \}.$$
From now on let $i$ be the minimal index, s.t. $m_i \neq 0$.

\begin{dfn}\label{definitionmi}
For $\bs \in S(\lam)$ denote by $M^{\bs}_i$ the set of minimal elements in $R_i^{\bs}$
with respect to the order $>$ (see \eqref{order}). Denote by $\bm_i^{\bs}$ the
tuple $m_{\beta} = 1$ if $\beta \in M_i^{\bs}$ and $0$ otherwise.
\end{dfn}

\begin{lem}
Let $\lam = \sum m_i \om_i$ and $i$ minimal with $m_i \neq 0$. 
If $\bs \in S(\lambda)$ then $\bm^{\bs}_i \in S(\om_i)$ and $\bs - \bm^{\bs}_i \in S(\lam - \om_i)$.
\end{lem}
\begin{proof}
We first note that the statement $\bm^{\bs}_i \in S(\om_i)$ follows from
Remark \ref{one}. Let us prove that $\bs - \bm^{\bs}_i \in S(\lam - \om_i)$.
For this we need to show that the conditions from \eqref{polytopeequation} for
$\bs - \bm^{\bs}_i$ are satisfied for all paths $\bp$.
Let $\bp=(p(0),\dots,p(k))$. Let $p(0)=\al_a$. Then we know that
$\sum_{l=0}^k s_{p(l)}\le m_a+\dots +m_b$, where $p(k)=\al_b$ if $b<n$ and
$p(k)=\al_{j,\ol j}$ if $b=n$ ($j\ge a$). The cases $b<i$ or $a>i$ are trivial.
So we assume $a\le i$ and $b\ge i$.
If $M_i^\bs \cap \bp\ne\emptyset$, then
\[
\sum_{l=0}^k (s-\bm_i^\bs)_{p(l)}\le m_a+\dots +m_b-1.
\]
Now assume that $M_i^\bs \cap \bp=\emptyset$. Let $l$ be the minimal number such that
$s_{p(l)}>0$. Then there exists $\al\in M_i^\bs$ such that $\al<p(l)$.
Therefore there exists a path $\bp'$ containing $\al$, $p(l), \dots, p(k)$.
We note that
\[
m_i+\dots +m_b\ge \sum_{l\ge 0} s_{p'(l)}> \sum_{l\ge 0} s_{p(l)}.
\]
Therefore, $\sum_{l\ge 0} (s-\bm_i^\bs)_{p(l)}\le m_i+\dots +m_b$.
\end{proof}

\subsection{Proof of the main theorem}
In the following we write $V^a(\la)$ for the associated graded module $gr V(\la)$.
Denote by $V^a(\la,\mu)\hk V^a(\la)\T V^a(\mu)$ the $S(\fn^-)$-submodule
generated by the tensor product $v_\la\T v_\mu$ of the highest weight vectors.

\begin{thm}\label{maintheorem}
\begin{itemize}
\item[i)] The vectors $f^\bs v_\la$, $\bs\in S(\la)$ form a basis of  $V^a(\la)$,
\item[ii)] Let $V^a(\la)=S(\fn^-)/I(\la)$. Then
$I(\la)=S(\fn^-)(\U(\n)\circ R)$, where
$$
R=\mathrm{span}\{ f_{\al_{i,j}}^{m_i+\dots +m_j+1}, 1\le i\le j\le n-1,\
f_{\alpha_{i, \ol{i}}}^{m_i+\dots + m_n+1}, 1\le i\le n\}.
$$
\item[iii)] The $S(\fn^-)$ modules $V^a(\la,\mu)$ and $V^a(\la+\mu)$ are
isomorphic.
\end{itemize}
\end{thm}

The proof of the theorem is by an inductive procedure. We know that part i)
of the theorem holds for all fundamental weights. For a dominant weight $\lam=\sum_i a_i\om_i$
denote by $\vert \la\vert=\sum a_i$ the sum of the coefficients. A first step in the proof
is the following proposition:

Let $\la$ be a dominant weight, and let $i$ be the minimal number such that
$(\la,\al_i)\ne 0$.
\begin{prop}\label{mainprop}
The vectors $f^\bs (v_{\la-\om_i}\T v_{\om_i})$, $\bs\in S(\la)$,
are linearly independent in $V^a(\la-\om_i)\T V^a(\om_i)$,
and the vectors $f^\bs (v_{\la})$, $\bs\in S(\la)$, form a basis
for $V^a(\la)$.
\end{prop}
\begin{proof}
The proof is by induction on $\vert\la\vert$. If $\la$ is a fundamental weight,
then the first part of the claim makes no sense and the second part is true.

So assume now $\vert\la\vert\ge 2$ and assume that the second part of the proposition holds for all
dominant weights $\mu$ such that $\vert\mu\vert <\vert\la\vert$. We prove now the first
part of the proposition for $\la$.

Assume that there exists some vanishing linear combination
\begin{equation}\label{lcomb}
\sum_{\bs\in S(\la)} c_\bs f^\bs (v_{\la-\om_i}\T v_{\om_i})=0.
\end{equation}
We will prove that $c_\bs=0$ for all $\bs$.

Recall first that we have an order $\ord$ on the set of ${\tt C}_n$-multi-exponents
(see Definition \ref{sp<}) such that if $\bt\notin S(\la)$ then
\[
f^\bt v_\la=\sum_{\substack{\bt\ord\bs\\ \bs\in S(\la)}} d_\bs f^\bs v_\la.
\]
Another important ingredient will be the elements $\bm^\bs_i$ (Definition~\ref{definitionmi}).
Recall that $i$ is minimal such that $(\la,\al_i)\ne 0$,
$\bm^\bs_i\in S(\om_i)$ and $\bs-\bm^\bs_i\in S(\la-\om_i)$.

The proof is by contradiction. Assume that $c_\bs\not=0$ for some $\bs$.
In the following we fix
such an element $\bs\in S(\la)$ and we assume without loss of generality that
$c_\bt=0$ for all $\bt\ord\bs$.

The vector space $V^a(\la-\om_i)\T V^a(\om_i)$ has a basis given by the elements
$f^\bfa v_{\la-\om_i}\T f^\bfb v_{\om_i}$, $\bfa\in S(\la-\om_i)$, $\bfb\in S(\om_i)$.
For all $\bt\in S(\la)$ such that $c_\bt\not=0$ in $(\ref{lcomb})$ we express
$f^\bt (v_{\la-\om_i}\T v_{\om_i})$ as a linear combination of these basis elements,
i.e., we will write
$$
f^\bt (v_{\la-\om_i}\T v_{\om_i})= \sum_{\substack{\bfa\in S(\la-\om_i)\\ \bfb\in S(\om_i)}} K^\bt_{\bfa,\bfb}
f^\bfa v_{\la-\om_i}\T f^\bfb v_{\om_i}.
$$
In the next step we show that $K^\bt_{\bs-\bm^\bs_i,\bm^\bs_i}=0$ for all $\bt\not=\bs$ and $K^\bs_{\bs-\bm^\bs_i,\bm^\bs_i}\not=0$.

Using the rules for the action on a tensor product we see:
\begin{equation}
f^\bs (v_{\la-\om_i}\T v_{\om_i})=C f^{\bs-\bm^\bs_i} v_{\la-\om_i}\T f^{\bm^\bs_i} v_{\om_i}
+\sum_{\br_1+\br_2=\bs} p_{\br_1,\br_2}f^{\br_1} v_{\la-\om_i}\T f^{\br_2} v_{\om_i},
\end{equation}
where $C$ is a nontrivial constant (a product of binomial coefficients) and
$\br_1\ne \bs-\bm^\bs_i$, $\br_2\ne \bm^\bs_i$. The elements
$f^{\br_1} v_{\la-\om_i}$, $f^{\br_2} v_{\om_i}$ need not to be basis
elements, we discuss the several possible cases separately. First assume that
$\br_2\in S(\om_i)\setminus\{\bm^\bs_i\}$. Then
$f^{\br_1} v_{\la-\om_i}\T f^{\br_2} v_{\om_i}$ is a sum of basis elements of
the form $f^{\bfa} v_{\la-\om_i}\T f^{\br_2} v_{\om_i}$ where
$(\bfa,\br_2)\not=(\bs-\bm^\bs_i,\bm^\bs_i)$. For the same reason, if
$\br_1\in S(\la-\om_i)\setminus\{\bs-\bm^\bs_i\}$, then
$f^{\br_1} v_{\la-\om_i}\T f^{\br_2} v_{\om_i}$ is a sum of basis elements of
the form $f^{\br_1} v_{\la-\om_i}\T f^{\bfb} v_{\om_i}$ where
$(\br_1,\bfb)\not=(\bs-\bm^\bs_i,\bm^\bs_i)$.
If $\br_1\notin S(\la-\om_i)$ and $\br_2\notin S(\om_i)$, then
\[
f^{\br_1} v_{\la-\om_i}=\sum_{\substack{\br_1\ord\bfa\\ \bfa\in S(\la-\om_i)}}
e_\bfa f^\bfa v_{\la-\om_i}
\text{\ \ and\ \ }
f^{\br_2} v_{\om_i}=\sum_{\substack{\br_2\ord\bfb\\ \bfb\in S(\om_i)}} d_\bfb f^\bfb v_{\om_i}
\]
with some constants $e_\bfa, d_\bfb$.  But among the pairs $(\bfa,\bfb)$ the pair
$(\bs-\bm^\bs_i,\bm^\bs_i)$ can not appear, because
\[
(\bs-\bm^\bs_i)+\bm^\bs_i=\bs=\br_1+\br_2\ord \bfa+\bfb.
\]
Therefore the expression of $f^\bs (v_{\la-\om_i}\T v_{\om_i})$ as a sum of the basis elements
is of the form
$$
C f^{\bs-\bm^\bs_i} v_{\la-\om_i}\T f^{\bm^\bs_i} v_{\om_i}
+\sum_{\substack{\bfa\in S(\la-\om_i),\bfb\in S(\om_i)\\ (\bfa,\bfb)\not=(\bs-\bm^\bs_i,\bm^\bs_i)}}
p_{\bfa, \bfb}f^{\bfa} v_{\la-\om_i}\T f^{\bfb} v_{\om_i}
$$
and hence $K^\bs_{\bs-\bm^\bs_i,\bm^\bs_i}\not=0$.

Now let us consider a term
$f^\bt (v_{\la-\om_i}\T v_{\om_i})$, $\bt\ne\bs$, such that $c_\bt\not=0$ in $(\ref{lcomb})$.
We write again
\begin{equation}
f^\bt (v_{\la-\om_i}\T v_{\om_i})=\sum_{\br_1+\br_2=\bt} p_{\br_1,\br_2}f^{\br_1} v_{\la-\om_i}\T f^{\br_2} v_{\om_i},
\end{equation}
and express each of the terms $f^{\br_1} v_{\la-\om_i}\T f^{\br_2} v_{\om_i}$ as a sum of the basis elements
$$
f^{\br_1} v_{\la-\om_i}\T f^{\br_2} v_{\om_i}=\sum_{\bfa\in S(\la-\om_i),\bfb\in S(\om_i)}
q_{\bfa,\bfb}f^{\bfa} v_{\la-\om_i}\T f^{\bfb} v_{\om_i}.
$$
Recall that $\bfa$ is less than or equal to $\br_1$, and $\bfb$ is less than or equal to $\br_2$.
We claim that none of the couples $(\bfa,\bfb)$ occurring
with a nonzero coefficient $q_{\bfa,\bfb}$ is equal to $(\bs-\bm^\bs_i,\bm^\bs_i)$.
The proof is by contradiction: If $(\bfa,\bfb)=(\bs-\bm^\bs_i,\bm^\bs_i)$, then
$\br_1+\br_2=\bt$ is greater than or equal to
$\bfa+\bfb=\bs-\bm^\bs_i+\bm^\bs_i=\bs,$
which is not possible, because $c_\bt=0$ if $\bt\ord\bs$.
Hence $K^\bt_{\bs-\bm^\bs_i,\bm^\bs_i}=0$ for all $\bt\not=\bs$.

It follows that if we express each of the summands in $(\ref{lcomb})$ as a linear combination of the basis elements
$f^{\bfa} v_{\la-\om_i}\T f^{\bfb} v_{\om_i}$, $\bfa\in S(\la-\om_i),\bfb\in S(\om_i)$, then the term
$f^{\bs-\bm^\bs_i} v_{\la-\om_i}\T f^{\bm^\bs_i} v_{\om_i}$ occurs only once with a non-zero coefficient,
which is not possible unless $c_\bs=0$ in $(\ref{lcomb})$.
Hence all coefficients
vanish in the expression in $(\ref{lcomb})$, proving the linear independence.

To prove the second part of the proposition recall the degree filtration
$\U(\n^-)_s$ on $\U(\n^-)$:
$$
\U(\n^-)_s=\mathrm{span}\{x_1\dots x_l:\ x_i\in\n^-, l\le s\},
$$
and recall that for a dominant weight $\mu$ we set
$$
V(\mu)_s=\U(\n^-)_sv_\mu.
$$
Then $V^a(\mu)$ is the associated graded $S(\n^-)$-module. The tensor product
$V^a(\la-\om_i)\T V^a(\om_i)$ of the graded modules with grading
$$
V^a(\la-\om_i)\T V^a(\om_i)=\bigoplus_{k\ge 0}\big( \oplus_{k=\ell +m}(V^a(\la-\om_i))_\ell \T (V^a(\om_i))_m \big)
$$
is the associated graded module for the filtration
$$
\big(V(\la-\om_i)\T V(\om_i)\big)_k=\sum_{k=\ell+m} V(\la-\om_i)_\ell \T V(\om_i)_m.
$$
Recall  the total order on the set of positive roots.
We write $f^\bs\in U(\n^-)$ for $\bs\in S(\la)$
for the ordered product of the corresponding root vectors. The linear
independence of the vectors $f^\bs (v_{\la-\om_i}\T v_{\om_i})$, $\bs\in S(\la)$,
in $V^a(\la-\om_i)\T V^a(\om_i)$ implies the linear independence
of the vectors $f^\bs (v_{\la-\om_i}\T v_{\om_i})$, $\bs\in S(\la)$,
in $V(\la-\om_i)\T V(\om_i)$. Since these vectors are all contained
in the Cartan component $V(\la)\hk V(\la-\om_i)\T V(\om_i)$, we obtain
the inequality  $\vert S(\la)\vert \le \dim V(\la)$.
We know already that the vectors $f^\bs (v_{\la})$, $\bs\in S(\la)$, span $V^a(\la)$
(section~\ref{spanningsection}, the straightening law, so
$\vert S(\la)\vert \ge  \dim V^a(\la)= \dim V(\la)$ and hence:
$$
\vert S(\la)\vert= \dim V^a(\la).
$$
It follows that the vectors $f^\bs (v_{\la})$, $\bs\in S(\la)$,  are in fact a basis for $V^a(\la)$.
\end{proof}
Using the straightening law in section~\ref{spanningsection} we get as an immediate consequence:
\begin{cor}\label{relationcorollary}
Let $V^a(\la)=S(\fn^-)/I(\la)$. Then
$I(\la)=S(\fn^-)(\U(\n)\circ R)$, where
$$
R=\mathrm{span}\{ f_{\al_{i,j}}^{m_i+\dots +m_j+1}, 1\le i\le j\le n-1,\
f_{\alpha_{i, \ol{i}}}^{m_i+\dots + m_n+1}, 1\le i\le n\}.
$$
\end{cor}
Using the defining relations for $V^a(\la)$, it is easy to see that we
have a canonical surjective map $V^a(\la)\rightarrow V^a(\lam-\om_i,\om_i)$
sending $v_\la$ to $v_{\la-\om_i}\T v_{\om_i}$. By Proposition~\ref{mainprop}
we know that the image of basis $\{f^\bs (v_{\la}),\bs\in S(\la)\}\subset V^a(\la)$
remains linearly independent and hence:
\begin{cor}
The $S(\fn^-)$ modules $V^a(\la-\om_i,\om_i)$ and $V^a(\la)$ are
isomorphic.
\end{cor}
\begin{proof} ({\it of Theorem~\ref{maintheorem}})
The first and second part of the theorem follow from
Proposition~\ref{mainprop} and Corollary~\ref{relationcorollary}. It remains
to prove the third part. As above, it is easy to see that we
have a canonical surjective map $V^a(\la+\mu)\rightarrow V^a(\lam,\mu)$
sending $v_{\la+\mu}$ to $v_{\la}\T v_{\mu}$.

The corollary above says that our theorem holds if $\mu=\om_i$. Iterating, we obtain that both
$V^a(\la,\mu)$ and $V^a(\la+\mu)$ sit inside the tensor product
\[
V^a(\om_1)^{\T(\la+\mu,\al_1)}\T\dots\T V^a(\om_n)^{\T(\la+\mu,\al_n)}
\]
as highest components (generated from the tensor product of highest weight vectors).
This proves the theorem.
\end{proof}

\section*{Acknowledgements}
The work of Evgeny Feigin was partially supported
by the Russian President Grant MK-281.2009.1, the RFBR Grants 09-01-00058
and NSh-3472.\\2008.2, by Pierre Deligne fund
based on his 2004 Balzan prize in mathematics.
The work of Ghislain Fourier was partially supported by the DFG project
``Kombinatorische Beschreibung von Macdonald und Kostka-Foulkes Polynomen``.
The work of Peter Littelmann was partially supported by the
priority program SPP 1388 of the German Science Foundation.

\end{document}